\documentclass[11pt,psamsfonts]{article}
\usepackage{amsmath,amssymb,latexsym}
\usepackage{amscd}
\usepackage[mathscr]{eucal}
\usepackage[english]{babel}
\usepackage[all]{xy}

\newcommand\qed{{\hspace*{\fill}Q.E.D.\vskip12pt plus 1pt}}

\newcommand\Pic[1]{\hbox{\rm Pic(}#1\hbox{\rm )}}

\def\Coker{\operatorname{Coker}}

\def\rank{\operatorname{rank}}

\def\lim{\operatorname{lim}}

\def\det{\operatorname{det}}
\def\codim{\operatorname{codim}}

\def\Pic{\operatorname{Pic}}

\def\Hom{\operatorname{Hom}}
\def\Ext{\operatorname{Ext}}

\def\id{\operatorname{id}}
\def\rank{\operatorname{rank}}
\def\Ker{\operatorname{Ker}}

\def\Coker{\operatorname{Coker}}

\def\calHom{{\mathscr H}om}
\def\calExt{{\mathscr E}xt}

\def\Hom{\operatorname{Hom}}
\def\Ext{\operatorname{Ext}}
\def\Pic{\operatorname{Pic}}
\def\det{\operatorname{det}}
\def\rank{\operatorname{rank}}

\newcommand\proof{\noindent{\em Proof.}\ \ }

\newtheorem{theorem}{Theorem}[section]

\newtheorem{corollary}[theorem]{Corollary}

\newtheorem{prop}[theorem]{Proposition}

\newtheorem{question}[theorem]{Question}

\newtheorem{definition}[theorem]{Definition}
\newtheorem{rem}[theorem]{Remark}

\newtheorem{pargrph}[theorem]{}
\newtheorem{examp}[theorem]{Example}

\newtheorem{MM}[theorem]{ }
\newtheorem{res}[theorem]{Remarks}

\makeatletter
\ifnum\@ptsize=0  \fi
\ifnum\@ptsize=2  \fi
\renewcommand{\qed}{\hfill $\square$}

\newenvironment{rem*}{\begin{rem}\em}{\end{rem}}
\newenvironment{rems*}{\begin{res}\em}{\end{res}}
\newenvironment{example*}{\begin{examp}\em}{\end{examp}}
\newenvironment{definition*}{\begin{definition}\em}{\end{definition}}
\newenvironment{question*}{\begin{question}\em}{\end{question}}
\newenvironment{MM*}{\begin{MM}\em}{\end{MM}}

\newenvironment{prgrph*}[1]{\indent\begin{pargrph}{\bf #1.}\em\
}{\end{pargrph}}

\begin{document}
\title{{Infinitesimal extensions of rank two vector bundles\\on submanifolds of small codimension}
}

\author{Lucian B\u adescu\footnote{{\em Author's address.} Universit\`a degli Studi di Genova, Dipartimento di Matematica,
Via Dodecaneso $35$, $16146$ Genova, Italy. Email: badescu@dima.unige.it}}

\date{}
\maketitle


\medskip\medskip

\begin{abstract}\noindent {\small\small Let $X$ be a submanifold of dimension $n$ of the complex projective space 
$\mathbb P^N$ ($n<N$), and let $E$ be a vector bundle of rank two on $X$ . If
$n\geq\frac{N+3}{2}\geq 4$ we prove a geometric criterion for the existence of an 
extension of $E$ to a vector bundle on the first order infinitesimal
neighborhood of $X$ in $\mathbb P^N$ in terms of the splitting of the 
normal bundle sequence of $Y\subset X\subset\mathbb P^N$, where 
$Y$ is the zero locus of a general section of a high twist of $E$. In the last section we show that
the universal quotient vector bundle  on the Grassmann
variety $\mathbb G(k,m)$ of $k$-dimensional linear subspaces of $\mathbb P^m$, with $m\geq 3$ and
$1\leq k\leq m-2$ (i.e. with $\mathbb G(k,m)$  not a projective space), embedded in any projective 
space $\mathbb P^N$, does not  extend to the first infinitesimal neighborhood of $\mathbb G(k,m)$ 
in $\mathbb P^N$ as a vector bundle.} \end{abstract}

\section*{Introduction}

Let $X$ be a submanifold of dimension $n$ of a complex projective manifold $P$ of dimension $N$, with $n<N$. For every $i\geq 0$ denote by $X(i)$ the $i$-th infinitesimal neighborhood of $X$ in $P$,  i.e. the subscheme of $P$ defined by the sheaf of ideals $\mathscr I_X^{i+1}$, where $\mathscr I_X$ is the sheaf of ideals  of $X$ in $\mathscr O_P$. Note that $X(0)=X$. Fix an $i\geq 0$; if $E$ is a vector bundle of rank $r$ on $X(i)$, a natural problem is to give criteria for the extendability of $E$ to the next infinitesimal neighborhood $X(i+1)$ as a vector bundle. The following general fundamental result  was proved by Grothendieck in 1960 (see \cite[\'Expos\'e III, Proposition 7.1, Page 85]{SGA1}):

\medskip

\noindent{\bf Theorem (Grothendieck)} {\em Under the above hypotheses and notation, assume that
\begin{equation}\label{grot} H^2(X,E\otimes E^{\vee}\otimes {\bf S}^{i+1}(N^\vee_{X|P}))=0,\end{equation}
where for every $j\geq 1$, ${\bf S}^j(N^\vee_{X|P})=\mathscr I^j_X/\mathscr I_X^{j+1}$
is the $j$-th symmetric power of the conormal bundle  $N^\vee_{X|P}=\mathscr I_X/\mathscr I_X^2$ of $X$ in $P$. Then $E$ can be extended to a vector bundle $\mathscr E$ on $X(i+1)$. If moreover $H^1(X,E\otimes E^{\vee}\otimes {\bf S}^{i+1}(N^\vee_{X|P}))=0$ then this extension is also unique up to isomorphism.}

\medskip

If in Grothendieck's theorem  above  $X$ is a curve and $E$  a vector bundle on $X$ then the vanishing \eqref{grot} is automatically fulfilled,  so that  $E$ can be extended to a vector bundle $\mathscr E_i$ on $X(i)$ for every $i\geq 1$. Note also that the vanishing \eqref{grot} is only a sufficient condition for the extendability of the vector bundle $E$ in Grothendieck's theorem. 

The aim  of this paper is twofold. First we prove, in the spirit of the paper \cite{EGPS} of Ellingsrud, Gruson, Peskine and Str\o mme, a necessary and sufficient geometric criterion  for extending a vector bundle $E$ of rank two on $X$ to a vector bundle $\mathscr E$ on the first infinitesimal neighborhood $X(1)$ of $X$ in $P$, when $P$ is the $N$-dimensional complex projective space $\mathbb P^N$ and $X$ is a submanifold of small codimension in 
$P=\mathbb P^N$, but without assuming the vanishing \eqref{grot} for $i=0$. In this paper  ``small codimension'' will mean that the inequalities $n\geq\frac{N+3}{2}\geq 4$ are satisfied. For example if $n=4$,  $X$ is a smooth hypersurface in $\mathbb P^5$, and if $n=5$, $X$ is either a smooth hypersurface in $\mathbb P^6$, or a $2$-codimensional submanifold in $\mathbb P^7$, and so on. We prove that a
vector bundle $E$ of rank two on $X$ can be extended to a vector bundle on $X(1)$ if and only if $E$ satisfies the condition
$({\bf P}_E^2)$ stated at the beginning of Section 2 (see Theorem \ref{a} below for the 
precise formulation). This condition involves the splitting of the canonical exact sequence of normal bundles
$$0\to N_{Y|X}=E(l)|_Y\to N_{Y|\mathbb P^N}\to N_{X|\mathbb P^N}|_Y\to 0,$$
where $Y$ is the zero locus of a general section of $E(l)$ for $l\gg 0$. This is done by first interpreting the splitting of the above exact sequence of normal bundles (via a generalization of a key lemma of \cite{EGPS} given in \cite{Rep}), and then by using a generalized form of the Hartshorne--Serre correspondence (Theorem \ref{serre-locale} below, whose proof was written jointly with E. Arrondo).
The second aim of this paper is to prove Theorem \ref{c} below, which asserts that the universal quotient vector bundle  of the Grassmann variety $\mathbb G(k,m)$ of linear subspaces of dimension $k$ in $\mathbb P^m$ (with $1\leq k\leq m-2$) never extends as a vector bundle to the first infinitesimal neighborhood of $\mathbb G(k,m)$ with respect to any projective embedding of $\mathbb G(k,m)$. 

The paper is organized as follows. In Section 1 we recall some known results needed in the next sections. In Section 2 we prove Theorem \ref{a} and in Section 3, Theorem \ref{c}.
\medskip

As a motivation of this paper, let me first recall the following beautiful result:

\medskip 

\noindent{\bf Theorem (Griffiths-Harris \cite{GrHa}, cf. also \cite{HaHu}, cf. also \cite{EGPS}}) 
{\em  Let $X$ be a smooth projective complex surface embedded in $\mathbb P^n$
$(n\geq 3)$ as a complete intersection. Let $Y$ be a smooth connected curve in $X$
such that the canonical exact sequence of normal bundles
$$0\to N_{Y|X}\to N_{Y|\mathbb P^n}\to N_{X|\mathbb P^n}|_Y\to 0$$
splits. Then there is a hypersurface $H$ of $\mathbb P^n$ such that
$Y=X\cap H$ $($scheme-theoretically$)$.}

\medskip\medskip

The crucial step of the (short and very elegant) proof of this result given in \cite{EGPS} is to show that the normal bundle $N_{Y|X}$ of $Y$ on $X$ can be extended to a line bundle of the first infinitesimal $X(1)$ of $X$ in $\mathbb P^n$. Instead, the proofs of \cite{GrHa} and \cite{HaHu} make use of the theory of infinitesimal variation of Hodge structures. The proof of Theorem \ref{a} below (which also involves the splitting of certain canonical exact sequences of normal bundles) makes use of Grothendieck-Lefschetz theory  plus a generalized form of Hartshorne-Serre correspondence (Theorem \ref{serre-locale} below) in order to extend certain rank two vector bundles on a small-codimensional submanifold $X$ of  $\mathbb P^n$ to rank two vector bundles on the first infinitesimal neighborhood $X(1)$ of $X$ in $\mathbb P^n$.

\medskip\medskip 

{\em Unless otherwise stated, throughout this paper we shall use the standard terminology and notation. All the algebraic varieties or schemes considered  are defined over the field $\mathbb C$ of complex numbers.}
\medskip

{\bf Acknowledgment.} I am grateful to Giorgio Ottaviani for having explained to me how the vanishing \eqref{d0} in Section 3 (needed to conclude the proof of Theorem \ref{c})  is a consequence of 
a general result of Ottaviani-Rubei (see \cite[Theorem 6.11] {Ott}). I also want to thank the referee for suggesting some improvements of the presentation and for providing a list of typos.

\section{Background material}

In this section we recall some known results that will be used in the next two sections.

\begin{prop}[Bertini--Serre, see \cite{F}, Appendix B9]\label{t1} Let $E$ be a vector bundle
of rank $r$ on an algebraic variety $X$ over $k$. Assume that $V$ is
a finite dimensional $k$-vector subspace of $H^0(X,E)$ whose sections generate
$E$. Then there is a non-empty Zariski open
subset $V_0$ of $V$ such that $\codim_XZ(s)\geq\min\{r,\dim X+1\}$ for every $s\in V_0$,
where $Z(s)$ denotes the zero locus of $s$ $($in particular, $Z(s)=\varnothing$  if $r>\dim X)$.\end{prop}

\begin{theorem}[Kodaira--Le Potier vanishing theorem \cite{LP}]\label{t2} Let $E$ be an ample vector bundle of rank $r$ on a smooth projective $n$-dimensional variety. Then $H^i(X,E^{\vee})=0$ for every 
 $i\leq n-r$, where $E^{\vee}$ is the dual of $E$.\end{theorem}

\begin{theorem}[Sommese \cite{S}]\label{t3} Let $E$ be an ample vector bundle of rank $r$ on a smooth projective $n$-dimensional variety such that $n-r\geq 2$. Let $s\in H^0(X,E)$ be a global section. Then the  zero locus $Y:=Z(s)$ is connected and nonempty of dimension $\geq n-r$. Assume moreover that $Y$ is smooth and
 $\dim Y=n-r$. Then the canonical restriction map  $\Pic(X)\to\Pic(Y)$ of Picard groups is an isomorphism if $n-r\geq 3$, and injective with torsion-free cokernel if $n-r=2$.\end{theorem}
 
\begin{theorem}[Barth--Larsen \cite{Lr}]\label{t4} Let $X$ be a smooth closed subvariety of dimension $n$ of $\mathbb P^N$. Then the canonical restriction map $\Pic(\mathbb P^N)\to\Pic(X)$ is an isomorphism if $n\geq\frac{N+2}{2}$ and is injective with torsion-free cokernel if $n=\frac{N+1}{2}$.\end{theorem}
 
\begin{theorem}[Van de Ven \cite{VdV}]\label{t5} Let $X$ be a smooth closed subvariety of dimension $\geq 1$ of $\mathbb P^N$. Then the canonical exact sequence of tangent bundles
$$0\to T_X\to T_{\mathbb P^N}|_X\to N_{X|\mathbb P^N}\to 0$$
splits if and only if $X$ is a linear subspace of $\mathbb P^N$.\end{theorem}
  
\begin{theorem}[\cite{EGPS}, \cite{Br} if $\codim_XY=1$ and \cite{Rep} if $\codim_XY>1$]\label{t6} Let $P$, $X$ and $Y$ be three smooth projective irreducible varieties such that $Y\subsetneq X\subsetneq P$ and $\dim Y\geq 1$. Set $r:=\codim_XY$. Then the canonical exact sequence of normal bundles 
$$0\to N_{Y|X}\to N_{Y|P}\to N_{X|P}|_Y\to 0$$
splits if and only if there exists a closed subscheme $Y'$ of the first infinitesimal neighborhood $X(1)$ of $X$ in $P$ such that $Y'$ is a local complete intersection of codimension $r$ in $X(1)$ and $Y'\cap X=Y$ $($scheme theoretically in $X(1)$, i.e. $\mathscr I_{Y'}+\mathscr I_X=\mathscr I_Y$, where 
$\mathscr I_{Y'}$, $\mathscr I_X$ and $\mathscr I_Y$ are the ideal sheaves of $\,Y'$, $X$ and $Y$ in $\mathscr O_{X(1)}$ respectively$)$.\end{theorem}

 Note  that the fact that $Y'$ is a local complete intersection in $X(1)$ of codimension $r$ is the essential part of the conclusion in Theorem \ref{t6}.

\section{Infinitesimal extensions of rank two vector bundles }

In this section we shall prove a geometric criterion for the extendability of a vector bundle $E$ of rank $2$ on a small-codimensional submanifold $X$ of $\mathbb P^N$ to a vector bundle $\mathscr E$ on $X(1)$ (Theorem \ref{a} below). 
We start (more generally) with a submanifold $X$ of $\mathbb P^N$ of dimension $n$ and with a vector bundle  $E$  a rank $r$ on $X$, with 
$1\leq r\leq n-1$. Then consider the following condition on the triple $(\mathbb P^N,X,E)$:
\begin{enumerate}
\item[$({\bf P}_E^r)$] {\em There exists an integer $l_0>0$  such that for every $l\geq l_0$ there exists a section $s=s_l\in H^0(E(l))$
whose zero locus $Y:=Z(s)$ is a smooth $r$-codimensional subvariety of $X$ such that the following canonical exact sequence of normal bundles
\begin{equation}\label{normali}
0\to N_{Y|X}=E(l)|_Y\to N_{Y|\mathbb P^N}\to N_{X|\mathbb P^N}|_Y\to 0
\end{equation}}splits.\end{enumerate}

\begin{prop}\label{a0} With the above notation, let $E$ be a vector bundle of rank $r$, with $1\leq r\leq n-1$,  on an $n$-dimensional submanifold
$X\subset\mathbb P^N$. If there exists a vector bundle $\mathscr E$ on $X(1)$ which extends $E$ then there exists an integer $l_0>0$ such that for every $l\geq l_0$ and for every section $s\in H^0(E(l))$
whose zero locus $Y$ is smooth $r$-codimensional in $X$, the exact sequence \eqref{normali} splits. In particular, condition $({\bf P}_E^r)$ above holds true.
\end{prop}

\begin{proof} Consider the exact sequence
$$0\to F\to\mathscr E\to \mathscr E|_{X}=E\to 0,$$
where $F:=\Ker(\mathscr E\to E)$. 
Since by a well-known theorem of Serre $H^1(X(1),F(l))=0$ for $l\gg 0$,  the map 
$H^0(X(1),\mathscr E(l))\to H^0(X,E(l))$ 
is surjective for $l\gg 0$. Moreover, enlarging $l$ enough, we can also assume that the vector bundle $E(l)$ is ample and generated by its global sections. Let $s\in H^0(X,E(l))$ be a global section whose zero locus 
$Y:=Z(s)$ is smooth and $(n-r)$-dimensional (indeed, since $E(l)$ is generated by its global sections, by Proposition \ref{t1}  a general section of $E(l)$ satisfies this condition). Moreover, by Theorem \ref{t3},  $Y$ is also connected, and hence irreducible because $Y$ is smooth. Then the section $s$ lifts to a global section $s'\in H^0(X(1),\mathscr E(l))$. If $Y'$ denotes the zero locus of $s'$ it follows that  $Y'\cap X=Y$ (scheme-theoretic intersection in $X(1)$). Moreover, $Y'$ is a local complete intersection of codimension $r$ in $X(1)$. Then by Theorem \ref{t6} above we conclude that the exact sequence \eqref{normali} splits.
\qed\end{proof}

To prove the main result of this section we need the following generalization of the so-called Hartshorne--Serre correspondence:

\begin{theorem}[Generalized Hartshorne--Serre correspondence] \label{serre-locale}
Let $\mathscr X$ be an arbitrary irreducible algebraic scheme  $($not necessarily reduced$)$ over a  field $k$, and let $\mathscr Y\subset \mathscr X$  be a local complete
intersection subscheme of $\mathscr X$ of codimension two. Assume that the determinant of
the normal bundle $N_{\mathscr Y|\mathscr X}$ of
$\mathscr Y$ in $\mathscr X$ extends to a line bundle $L$ on $\mathscr X$ such that
$H^2(\mathscr X,L^{-1})=0$. Then there exists a vector bundle $\mathscr E$ of rank two
on $\mathscr X$ and a global section $t\in H^0(\mathscr X,\mathscr E)$ such that $\det(\mathscr E)=L$ and
$Z(t)=\mathscr Y$, i.e. the zero locus of $t$ is $\mathscr Y$ $($scheme-theoretically$)$.
If moreover $H^1(\mathscr X,L^{-1})=0$ then the pair $(\mathscr E,t)$ is also unique up to isomorphism.
\end{theorem} 

\proof In the case when $\mathscr X$ is smooth the result is well-known
(see e.g. \cite{A}). For the lack of an appropriate reference we sketch a proof in this generality.
Let $\mathscr I_{\mathscr Y}$ denote the sheaf of ideals of $\mathscr Y$ in $O_{\mathscr X}$ and consider the spectral sequence (see e.g. \cite[Proposition (IV, 2.4]{AK}) 
$$E^{p,q}_2:=H^p(\mathscr X,\calExt_{\mathscr O_{\mathscr X}}^q(\mathscr I_{\mathscr Y}\otimes
L,\mathscr O_{\mathscr X}))\Longrightarrow
 E^n:=\Ext^n_{\mathscr O_{\mathscr X}}(\mathscr I_{\mathscr Y}\otimes
L,\mathscr O_{\mathscr X}),$$ 
which yields the exact sequence in low degrees:
\begin{equation}\label{esatta-spettrale}
0\to E_2^{1,0}\to E^1\to E_2^{0,1}\to E_2^{2,0}.
\end{equation}
Since $\mathscr Y$ is a local complete intersection in $\mathscr X$ of codimension two, 
$\calHom_{\mathscr O_{\mathscr X}}(\mathscr I_{\mathscr Y}\otimes  L,\mathscr O_{\mathscr X})\cong L^{-1}$,
so that our hypothesis that $H^2(\mathscr X,L^{-1})=0$ implies $E_2^{2,0}=0$. Thus
\eqref{esatta-spettrale} yields a canonical surjection
\begin{equation}\label{suriezione}
\Ext^1_{\mathscr O_{\mathscr X}}(\mathscr I_{\mathscr Y}\otimes L,\mathscr  O_{\mathscr X})\to
H^0(\mathscr X,\calExt_{\mathscr O_{\mathscr X}}^1(\mathscr I_{\mathscr Y}\otimes
L,\mathscr O_{\mathscr X})).
\end{equation}
On the other hand, the long exact cohomology sequence obtained by applying $\calHom_{\mathscr O_{\mathscr X}}(-,\mathscr O_{\mathscr X})$ to the short exact sequence
$0\to\mathscr I_{\mathscr Y}\otimes L\to L\to\mathscr O_{\mathscr Y}\otimes L\to 0$ 
immediately yields 
\begin{equation}\label{isomf}\calExt_{\mathscr O_{\mathscr X}}^1(\mathscr I_{\mathscr Y}\otimes L,\mathscr O_{\mathscr X})\cong 
\calExt_{\mathscr O_{\mathscr X}}^2(\mathscr O_{\mathscr Y}\otimes L,\mathscr O_{\mathscr X}).\end{equation}
Since $\mathscr Y$ is a local complete intersection in $\mathscr X$, by \cite{AK}, 
Theorem (I, 4.5) we infer that there is an isomorphism
\begin{equation}\label{firenze}\calExt_{\mathscr O_{\mathscr X}}^2(\mathscr O_{\mathscr Y}\otimes
L,\mathscr O_{\mathscr X})\cong\det(N_{\mathscr Y|\mathscr X})\otimes L^{-1}|\mathscr Y\cong\mathscr O_{\mathscr Y},
\end{equation}
because by assumption,
$L|_{\mathscr Y}=\det(N_{\mathscr Y|\mathscr X})$.  Therefore the  target of the surjection
\eqref{suriezione} becomes $H^0(\mathscr Y,\mathscr O_{\mathscr Y})=\Hom_{\mathscr O_{\mathscr Y}}(\mathscr O_{\mathscr Y},\mathscr O_{\mathscr Y})$. Hence the
identity map lifts  to an extension
\begin{equation}\label{estensione}
\begin{CD}0@>>>\mathscr O_{\mathscr X}@>t>>\mathscr E@>>>\mathscr I_{\mathscr Y}\otimes L@>>>0,\end{CD}
\end{equation}
which produces a rank-two coherent sheaf $\mathscr E$ on $\mathscr X$. We shall prove that $\mathscr E$ is actually locally free. To show this it is enough to prove that 
\begin{equation}\label{estensione2}\calExt_{\mathscr O_{\mathscr X}}^1(\mathscr E,\mathscr O_{\mathscr X})=0.\end{equation}
Indeed, the problem being local this follows from \cite[Lemma 5.1.2 and its proof, pages 98--99]{OSS}. 
To prove  \eqref{estensione2}, observe that  \eqref{estensione} yields the following exact sequence  
$${\begin{CD}\mathscr O_{\mathscr X}\cong\calHom_{\mathscr O_{\mathscr X}}(\mathscr O_{\mathscr X},\mathscr O_{\mathscr X})@>{\varphi}>>
\calExt_{\mathscr O_{\mathscr X}}^1(\mathscr I_{\mathscr Y}\otimes L,\mathscr O_{\mathscr X})@>>>\calExt_{\mathscr O_{\mathscr X}}^1(\mathscr E,\mathscr O_{\mathscr X})@>>>0\end{CD}}.$$
Thus \eqref{estensione2} is equivalent to the surjectivity of the map $\varphi$. 
But by \eqref{isomf} and \eqref{firenze} we get an isomorphism
$\calExt_{\mathscr O_{\mathscr X}}^1(\mathscr I_{\mathscr Y}\otimes L,\mathscr O_{\mathscr X})\cong
\mathscr O_{\mathscr Y}.$  
It follows  that  the map $\varphi$ is 
identified with the canonical surjection $\mathscr O_{\mathscr X}\to\mathscr O_{\mathscr Y}$. This finishes the existence part of the theorem.

Now the condition  that $\det(\mathscr E)\cong L$  follows immediately. Indeed, restricting the exact sequence \eqref{estensione} to $\mathscr X\setminus \mathscr Y$ and taking the determinant we get
that $\det(\mathscr E)|_{\mathscr X\setminus \mathscr Y}\cong L|_{\mathscr X\setminus \mathscr Y}$. Since the ideal of $\mathscr Y$ in $\mathscr X$ is locally generated by a regular sequence of length $2$, a standard argument based on Local Cohomology \cite{LC}  implies that  $\det(\mathscr E)\cong L$. Moreover, the condition that $Z(t)= Y$  follows directly from \eqref{estensione} and from the definition of the zero locus.

Finally assume that  $H^1(\mathscr X,L^{-1})=0$. This means that $E^{1,0}_2=0$ in the exact sequence \eqref{esatta-spettrale}, hence the surjective map
 \eqref{suriezione} is also injective. This yields the uniqueness 
of $\mathscr E$ (up to isomorphism), concluding the proof of the theorem.
\qed

\begin{rem} {\em The above proof of Theorem \ref{serre-locale} came out from a discussion  with Enrique Arrondo.}\end{rem}

The main result of this section is a sort of converse of Proposition \ref{a0} for rank two vector bundles on small-codimensional submanifolds in $\mathbb P^N$.
Precisely, we prove the following:

\begin{theorem}\label{a} Let $X\subset\mathbb P^N$ be a smooth $n$-dimensional
subvariety, with $n\geq\frac{N+3}{2}\geq 4$.  Let 
$E$ be a rank two vector bundle on $X$ which satisfies condition
$({\bf P}_E^2)$ above. Then $E$ can be extended  to a rank two vector bundle $\mathscr E$ on  the 
first infinitesimal neighborhood $X(1)$ of $X$ in $\mathbb P^N$.
\end{theorem}

\begin{proof} Assume first $n\geq 5$. It is clear that $E$ extends to a (rank-two) vector bundle on $X(1)$ if and only if $E(l)$ does, so that we can replace 
$E$ by a sufficiently high twist $E(l)$; in particular we may assume that $E$ is ample and generated by its global
sections. Since $X$ is in the range of Barth-Larsen theorem (Theorem \ref{t4}), its Picard
group is generated by the  class of $\mathscr O_X(1)$. In particular, since $E$ is ample, there exists an $m>0$ such that $\det(E)\cong{\mathscr O}_X(m)$. Replacing again $E$ by $E(l)$ with $l\gg 0$ if necessary, we may also assume that $m\gg 0$. Then the exact sequence
$$0\to N_{X|\mathbb P^N}^\vee\to\mathscr O_{X(1)}\to\mathscr O_X\to 0$$
yields the cohomology sequences ($i=1,2$)
$$H^i(N_{X|\mathbb P^N}^\vee(-m))\to H^i(\mathscr O_{X(1)}(-m))\to H^i(\mathscr O_X(-m)).$$
By \cite[\'Expos\'e XII, Corollaire 1.4]{Gro},  the first and the last vector space are zero for $i=1,2$ because $X$ is smooth of dimension $\geq 3$,  $N_{X|\mathbb P^N}$ is a vector bundle, and $m\gg 0$. Therefore we get:
\begin{equation}\label{secondo-vanishing}
H^i({\mathscr O}_{X(1)}(-m))=0 \;\;\text{for}\;\;i=1,2\;\;\text{and}\;\;m\gg 0.
\end{equation}
(Alternatively, a standard small argument  shows that the projective scheme $X(1)$ is locally Cohen-Macaulay, and then
\eqref{secondo-vanishing} follows directly from \cite[\'Expos\'e XII, Corollaire 1.4]{Gro}.) 

Replacing $E$ by $E(l)$ with $l\gg 0$, condition $({\bf P}_E^2)$ implies that there is a section $s\in H^0(X,E)$ whose zero locus $Y:=Z(s)$ is a smooth $2$-codimensional subvariety of $X$, and the canonical exact sequence \eqref{normali} splits.
Moreover, by Theorem \ref{t3} $Y$ is also connected (cf. also a subsequent more general connectivity theorem of Fulton-Lazarsfeld \cite{FL}).

Now by condition $({\bf P}_E^2)$ again and Theorem \ref{t6}, there exists a $2$-codimensional local complete intersection subscheme
$Y'$ of $X(1)$ such that $Y'\cap X=Y$ scheme-theoretically (i.e. $\mathscr I_{Y'}+\mathscr I_X=\mathscr I_Y$). We will show that $Y'$ is the zero locus of a section of a a rank two vector bundle on $X(1)$ by applying Theorem \ref{serre-locale}, with $\mathscr X:=X(1)$ and $\mathscr Y:=Y'$. 

In order to do this we will first show that $\det(N_{Y'|X(1)})$ 
extends to a line bundle $L$ on $X(1)$ such that
$H^i(X(1),L^{-1})=0$ for $i=1,2$. 
In this sense, consider the following commutative square of restriction maps
\begin{equation}\label{Picard}
\begin{CD}
\Pic(X(1)) @>\alpha>> \Pic(X) \\
@V\beta VV @VV\gamma V \\
\Pic(Y')@>\delta>>\Pic(Y).\\
\end{CD}
\end{equation}


\noindent {\em Claim.} If $n\geq 5$, all the maps in diagram \eqref{Picard} are isomorphisms, and if  $n=4=\frac{N+3}{2}$ (i.e. $X$ is a smooth hypersurface in $\mathbb P^5$), $\alpha$ is an isomorphism and $\gamma$, $\beta$ and $\delta$ are injective. 

\medskip

Let us prove the claim. The fact that $\alpha$ is an isomorphism if $n\geq 4$ follows immediately from the truncated exponential exact sequence 
$$0\to N^\vee_{X|\mathbb P^N}\to{\mathscr O}^*_{X(1)}\to{\mathscr O}^*_X\to1$$
and from  Theorem \ref{t2} (which implies in particular that $H^i(X,N^\vee_{X|\mathbb P^N})=0$, $i=1,2$, because $n\geq\frac{N+3}{2}\geq 4$). Here $\mathscr O^*_Z$ denotes the sheaf of multiplicative groups of regular nowhere vanishing functions on a scheme $Z$. 
Theorem \ref{t3} implies that the map $\gamma$ is an isomorphism because $n-\rank(E)=n-2\geq 4-2=2$.

Assume first $n\geq 5$.
At this point, since $\alpha$ and $\gamma$ are isomorphisms, the commutative diagram \eqref{Picard} implies that $\beta$ is injective and $\delta$ surjective. 
Therefore to finish the proof of the claim it is enough to show that the map $\delta$ is injective.
To do this, since the ideal sheaf of $Y$ in $Y'$ is 
square-zero, one still has the following truncated exponential sequence
$$0\to N_{Y|Y'}^{\vee}\to{\mathscr O}^*_{Y'}\to{\mathscr O}^*_Y\to 1.$$
On the other hand, by \cite[Remark 1.2 ii)]{Rep},  $N_{Y|Y'}\cong N_{X|\mathbb P^N}|_{Y}$, and in particular, $N_{Y|Y'}$ is an ample vector bundle because $N_{X|\mathbb P^N}$ is so. Therefore by Theorem \ref{t2} we get
$$H^1(Y,N_{Y|Y'}^{\vee})\cong H^1(Y,N^\vee_{X|\mathbb P^N}|_Y)\newline=0,$$ because $N_{X|\mathbb P^N}$ is ample, $\dim Y=n-2$, 
$\rank(N_{Y|Y'})=N-n$ and $N\leq 2n-3$.  Then the exact sequence
$$0=H^1(Y,N_{Y|Y'}^{\vee})\to\Pic(Y')\to\Pic(Y)$$
implies that $\delta$ is injective. 

If instead $n=4=\frac{N+3}{2}$ 
the injectivity of $\gamma$ follows from the last part of Theorem \ref{t3}.
The proof of the injectivity of $\delta$ when $n\geq 5$ works also if $n=4$. 
The claim is proved.

\medskip

Now, as $\dim X\geq\frac{N+3}{2}$, by Theorem \ref{t4}  the map $\Pic(\mathbb P^N)\to\mathbb\Pic(X)$ is an isomorphism.
Since $\mathscr I_{Y'}\subset\mathscr I_Y$ and $N^\vee_{Y'|X(1)}=\mathscr I_{Y'}/\mathscr I_{Y'}^2$ and $N^\vee_{Y|X}=\mathscr I_{Y}/\mathscr I_{Y}^2$ are vector bundles of the same rank,  $N_{Y'|X(1)}|_{Y}\cong N_{Y|X}=E|_{Y}$, hence   $\det(N_{Y'|X(1)})|_Y\cong\det(E)|_Y=\mathscr O_Y(m)$.  Since by the above claim  the map $\delta$ is injective (even an isomorphism if $n\geq 5$), the we get $\det(N_{Y'|X(1)})\cong {\mathscr O}_{Y'}(m)$, hence $L:=\mathscr O_{X(1)}(m)$ is the unique extension (up to isomorphism) of $\det(N_{Y'|X(1)})$ on $X(1)$. Then by \eqref{secondo-vanishing} we have
$H^i(X(1),L^{-1})=H^1(X(1),\mathscr O_{X(1)}(-m))=0$ for $i=1,2$ and for $m\gg0$.

Then by Theorem \ref{serre-locale}  applied to  $\mathscr X:=X(1)$ and $\mathscr Y=Y'$, there is  a  pair  
$(\mathscr E,t)$, with $\mathscr E$ a vector bundle of rank $2$ on $X(1)$ and a global section $t\in H^0(X(1),\mathscr E)$, uniquely determined up to isomorphism, such that:

 i) $\det(\mathscr E)=L$, and

 ii) $Z(t)=Y'$, i.e. the zero locus of $t$ is $Y'$ (scheme-theoretically).

Set $E':=\mathscr E|_X$ and $s':=t|_X\in H^0(X,E')$. Clearly, $\det(E')\cong\mathscr O_X(m)\cong\det(E)$. Moreover, as $\,Y'\cap X=Y$ (scheme-theoretically) we infer that $Z(s')=Y$. As $\det(N_{Y|X})\cong N_{Y|X}=\det(E|_Y)\cong\mathscr O_X(m)$, $\det(N_{Y|X})$ extends to  $\mathscr O_X(m)$  with $m>0$. Then a Serre
 vanishing we get $H^i(X,\mathscr O_X(-m))=0$ for $i=1,2$ and for $m\gg0$. In conclusion, $\det(E)$ and $\det(E')$ extend both on $X$ to $\mathscr O_X(m)$ and $Z(s)=Z(s')=Y$. Then by the uniqueness  part of Theorem \ref{serre-locale} there is an isomorphism 
$\varphi\colon E\to E'$ of vector bundles such that $\varphi(s)=s'$.  This implies that $\mathscr E|_{X}\cong E$, i.e. $\mathscr E$ is an infinitesimal extension of $E$.
 \qed\end{proof} 

\begin{rem}\label{pros} 
{\em A careful look at the proof of Theorem \ref{a} shows that this result is still true if one replaces condition $({\bf P}_E^2)$ above on the triple 
$(\mathbb P^N,X,E)$, with $r=\rank(E)=2$,  
 by the following (slightly) weaker one:}
\begin{enumerate}
\item[$({\bf P}_E^2)'$] There exists a  sequence of positive integers $l_0< l_1<l_2<\cdots$  such that for every $i\geq 0$ there exists a section 
$s_i\in H^0(E(l_i))$
whose zero locus $Y_i:=Z(s_i)$ is a smooth and $2$-codimensional in $X$ such that the following canonical exact sequence 
$$0\to N_{Y_i|X}=E(l)|_{Y_i}\to N_{Y_i|\mathbb P^N}\to N_{X|\mathbb P^N}|_{Y_i}\to 0$$
splits.
\end{enumerate}
\end{rem}

\section{Examples of infinitesimally non extendable  vector bundles}

Consider the Grassmann variety $\mathbb G(k,m)$ of
$k$-dimensional linear subspaces of $\mathbb P^m$, with $m\geq 3$ and
$1\leq k\leq m-2$ (hence $\mathbb G(k,m)$ is not a projective space). Then $\dim \mathbb G(k,m)=(k+1)(m-k)$. Let
$E$ denote the universal quotient bundle  of $\mathscr O_{\mathbb G(k,m)}^{\oplus m+1}$ (of rank $m-k$). 
Fix an arbitrary projective embedding $\mathbb G(k,m)\hookrightarrow\mathbb P^N$ (for example, the Pl\"ucker embedding
$i\colon X\hookrightarrow\mathbb P^{\binom{m+1}{k+1}-1}$), and denote by $X$ the image of $\mathbb G(k,m)$ in $\mathbb P^N$. 

\medskip

In this section we prove the following result:

\begin{theorem} \label{c} Under the above notation and hypotheses  the universal quotient vector bundle $E$ of $X\cong\mathbb G(k,m)$ $($with $1\leq k\leq m-2)$ cannot be extended to a vector bundle on the first infinitesimal
neighborhood $X(1)$ of $X$ in $\mathbb P^{N}$. \end{theorem}

\begin{proof} Assume by way of contradiction that there
would exist a vector bundle $\mathscr E$ on $X(1)$ such that
$\mathscr E|_{X}\cong E$. Tensoring by $\mathscr E$ the exact sequence
$$0\to N_{X|\mathbb P^N}^{\vee}\to\mathscr O_{X(1)}\to\mathscr O_{X}\to 0$$ 
and taking into account that
$\mathscr E\otimes N_{X|\mathbb P^N}^{\vee}\cong E\otimes
N_{X|\mathbb P^N}^{\vee}$ we get the exact sequence
\begin{equation}\label{d1}
0\to E\otimes N_{X|\mathbb P^N}^{\vee}\to\mathscr E\to E\to 0.
\end{equation}

Now assume for the moment that the following condition holds true
\begin{equation}\label{d}
H^1(X,E\otimes N_{X|\mathbb P^N}^{\vee})=0.
\end{equation}

Then \eqref{d1} and \eqref{d} imply that the restriction map $H^0(X(1),\mathscr E)\to H^0(X,E)$ is surjective. Considering the canonical surjection $\varphi\colon\mathscr O_{X}^{\oplus (m+1)}\twoheadrightarrow E$ given by $(s_0,s_1,\ldots,s_m)\in H^0(X,E)^{\oplus (m+1)}$, it follows that there exists an $(m+1)$-uple 
$(s'_0,s'_1,\ldots,s'_m)\in H^0(X(1),\mathscr E)^{\oplus (m+1)}$ such that $s'_i|X=s_i$, $i=0,1,\ldots,m$. Since $\varphi$ is surjective, the sections $s_0,s_1,\ldots,s_m$ generate $E$, hence by
Nakayama's Lemma the sections $s'_0,s'_1,\ldots,s'_m$
generate
$\mathscr E$. In other words, the surjection $\varphi$
lifts to a surjection $\varphi'\colon\mathscr O_{X(1)}^{\oplus (m+1)}\twoheadrightarrow\mathscr E$. Then by the
universal property of the Grassmann variety $X=\mathbb  G(k,m)$
 there exists a morphism of schemes
$\pi\colon X(1)\to X$ such that $\pi^*(E)=
\mathscr E$. Since $\mathscr E|_{X}=E$ it follows that $\pi$ is a retraction of the canonical embedding $X\hookrightarrow X(1)$.  By a well known result (see \cite{MP}, or also \cite[Lemma 6.2]{B}), this latter fact is equivalent with the splitting of the canonical exact sequence of tangent and normal bundles
\begin{equation}\label{e}0\to T_{X}\to T_{\mathbb P^N}|_X\to N_{X|\mathbb P^N}\to 0.\end{equation}
By Theorem \ref{t6} of Van de Ven, the splitting of \eqref{e}
implies that $X$ is a linear subspace of
$\mathbb P^N$, which is a contradiction (otherwise
$X$ would be isomorphic to a projective space). 

Now we prove \eqref{d}. We first claim that  \eqref{d} is equivalent with the following vanishing:
\begin{equation}\label{d0} H^0(E\otimes F)=0,\end{equation}
where $F$ is is defined in the following commutative diagram with exact rows and columns:
$${\small\small
\begin{CD}
{}   {}@. @. {0} @. {0} \\
 @.@. @VVV @VVV\\
{0} @>>> N_{X|\mathbb P^{N}}^{\vee} @>>> \Omega^1_{\mathbb P^{N}}|X @>>> {}\Omega^1_X @>>> {0}\\
 @. @V\id VV @VVV @VVV  \\
{0} @>>> N_{X|\mathbb P^{N}}^{\vee}  @>\varphi>> \mathscr O_X(-1)^{\oplus (N+1)} @>>>F:=\Coker(\varphi) @>>>{0}\\
 @.@. @VVV @VVV\\
 {} {}@.@. \mathscr O_X@>\id>> \mathscr O_X\\
@. @. @VVV @VVV\\
{}  {}@. @. {0} @. {0}
\end{CD}\;\;.}$$
\noindent The first row  in this diagram is the conormal sequence of $X$ and the second column is the Euler sequence restricted to $X$. Note that the sheaf $F$ coincides with ${\mathscr P}^1(\mathscr O_X(1))(-1)$, where ${\mathscr P}^1(\mathscr O_X(1))$ is the sheaf of first-order  principal
parts of $\mathscr O_X(1)$. 
Tensoring this diagram by $E$ we get the following commutative diagram with exact rows and columns
\begin{equation}\label{XXX}
\begin{CD}
{}   {}@. @. {0} @. {0} \\
 @.@. @VVV @VVV\\
{0} @>>> E\otimes N_{X|\mathbb P^{N}}^{\vee} @>>>E\otimes \Omega^1_{\mathbb P^{N}}|X @>>> {}E\otimes \Omega^1_X @>>> {0}\\
 @. @V\id VV @VVV @VVV  \\
{0} @>>> E\otimes N_{X|\mathbb P^{N}}^{\vee}  @>>> E(-1)^{\oplus (N+1)} @>>> E\otimes F@>>>{0}\\
 @.@. @VVV @VVV\\
 {} {}@.@. E@>\id>> E\\
@. @. @VVV @VVV\\
{}  {}@. @. {0} @. {0}
\end{CD}
\end{equation}

\noindent The second row of \eqref{XXX} yields the cohomology sequence
$$H^0(X,E(-1)^{\oplus (N+1)})\to H^0(X,E\otimes F)\to
  H^1(E\otimes N^{\vee}_{X|\mathbb P^N})\to H^1(E(-1)^{\oplus (N+1)}).$$
  
 \noindent By \cite[Corollary (4.11) and Theorem (4.17)]{Fu} (whose proofs are based on some vanishing results for flag manifolds of Kempf \cite{K}) we have $H^i(E(-1)^{\oplus(N+1)})=0$ for $i=0,1$ (since
$\Pic(X)\cong\mathbb Z$). 
Thus the canonical map $\delta\colon H^0(E\otimes F)\to H^1(X,E\otimes N_{X|\mathbb P^N}^{\vee})$  is an isomorphism, which proves the claim. 

Therefore it will be sufficient to prove \eqref{d0}.
But, as Giorgio Ottaviani kindly explained to me,  \eqref{d0} is a special case of a general result of Ottaviani-Rubei (see  
\cite[Theorem 6.11]{Ott}). Indeed, considering  the coboundary map $\delta'\colon H^0(E) \to H^1(E\otimes\Omega^1_X)$ associated to the last column of diagram \eqref{XXX}  as a quiver we infer that $\delta'\neq 0$. Moreover since $H^0(E)$ is the standard la representation (and hence irreducible), this implies that $H^0(E\otimes F)=0$.  Alternatively, the fact that $H^0(E)$ is irreducible was proved directly in \cite{Weh}. In this way the proof of Theorem \ref{c} is complete. 
\qed\end{proof}

\begin{rem} {\em In the special case of  Pl\"ucker embedding $X=\mathbb G(1,3)\hookrightarrow\mathbb P^5$, the normal bundle $N_{X|\mathbb P^5}$ is isomorphic to $\mathscr O_X(1)$, hence the vanishing \eqref{d} becomes $H^1(X,E(-1))=0$, and this follows directly from \cite[Corollary (4.11) and Theorem (4.17)]{Fu}.}\end{rem}

\begin{examp}[Submanifods of $\mathbb P^N$ of dimension $\frac{N+3}{2}$]\label{r2} {\em For every $m\geq 3$ consider the Pl\"ucker embedding $i'_m\colon \mathbb G(1,m)\hookrightarrow\mathbb P^{\binom{m+1}{2}-1}$ of the Grassmann variety of lines in $\mathbb P^m$, and set $X'_m:=i'_m(\mathbb G(1,m))$. As is well-known $X'_m$ is a $4$-defective subvariety of $\mathbb P^{\binom{m+1}{2}-1}$, meaning that there is a linear projection $\pi_{L_m}\colon\mathbb P^{\binom{m+1}{2}-1}\dasharrow\mathbb P^{4m-7}$ of center a linear subspace $L_m$ of dimension $\binom{m+1}{2}-(4m-7)-2$ of $\,\mathbb P^{\binom{m+1}{2}-1}$ which does not intersect $X'_m$ such that the restriction $\pi_{L_m}|X'_m\colon X'_m\to\pi_{L_m}(X'_m)$ is a biregular isomorphism (see \cite[Exercise 11.27, page 145]{Harr}). 
Therefore we may consider the projective embedding
$i_m:=\pi_{L_m}\circ i'_m\colon \mathbb G(1,m)\hookrightarrow\mathbb P^{4m-7}$. If we set $X_m:=\pi_{L_m}(X'_m)$, $n:=\dim X_m=\dim\mathbb G(1,m)=2(m-1)$ and $N:=4m-7$, it follows that $X_m$ is, via the projective embedding $i_m$, an $n$-dimensional closed subvariety of $\mathbb P^N$, with $n=\frac{N+3}{2}$. 
If $m=3$ or if $m=4$ the projective embeddings $i_m$ and  $i'_m$ coincide, i.e. $i_m$ is one of the Pl\"ucker embeddings $i'_3\colon\mathbb G(1,3)\hookrightarrow\mathbb P^5$ or
$i'_4\colon\mathbb G(1,4)\hookrightarrow\mathbb P^9$. Conversely, if $i_m$ and $i'_m$ coincide then $m=3$ or $m=4$.
In particular, Theorem \ref{a} applies to every rank two vector bundle on the submanifold $X_m$ of dimension $2(m-1)$ of $\mathbb P^{4m-7}$, with $m\geq 3$.}\end{examp}

\begin{corollary}\label{s} Let $X:=\mathbb G(1,m)$ be the Grassmann variety of lines in $\mathbb P^m$, with $m\geq 3$, and let $E$ be the universal rank two quotient vector bundle on $X$. Let $X\hookrightarrow\mathbb P^{4m-7}$ be any projective embedding of $X$ in $\mathbb P^{4m-7}$ $($see Example  $\ref{r2})$. Then there exists an  integer $l_0>0$ such that for every $l\geq l_0$ and for every section $s\in H^0(E(l))$  whose zero-locus 
$Y:=Z(s)$ is smooth of codimension $2$ in $X$, the exact sequence of normal bundles
$$0\to N_{Y|X}\to N_{Y|\mathbb P^N}\to N_{X|\mathbb P^N}|_Y\to0$$
never splits.\end{corollary}

\proof In this case $\dim X=2(m-1)=\frac{N+3}{2}$, with $N=4m-7$. Then the corollary follows from  Theorem \ref{c} and Theorem \ref{a} via Remark \ref{pros}.\qed

\end{document}